\documentclass{amsart}
 \usepackage{amssymb,amsmath,amsfonts,epsfig,latexsym}

 \newtheorem{theorem}{Theorem}[section]
\newtheorem{definition}[theorem]{Definition}
\newtheorem{proposition}[theorem]{Proposition}
\newtheorem{lemma}[theorem]{Lemma}
\newtheorem{corollary}[theorem]{Corollary}

\newtheorem*{KL}{Kedlaya's Lemma}

\theoremstyle{definition}
\newtheorem{remark}[theorem]{Remark}
\newtheorem{example}[theorem]{Example}

\def\Z{\ensuremath{\mathbb{Z}}}
\def\Q{\ensuremath{\mathbb{Q}}}

\def\R{\ensuremath{\mathbb{R}}}
\def\F{\ensuremath{\mathbb{F}}}

\def\m{\ensuremath{\mathfrak{m}}}

\def\<{\ensuremath{\langle}}
\def\>{\ensuremath{\rangle}}

\DeclareMathOperator{\ev}{ev}

\DeclareMathOperator{\Hom}{Hom}
\DeclareMathOperator{\init}{in}

\DeclareMathOperator{\Spec}{Spec}

\DeclareMathOperator{\supp}{supp}

\DeclareMathOperator{\Trop}{Trop}
\DeclareMathOperator{\TTrop}{\mathfrak{Trop}}

\begin{document}

\title{Fibers of tropicalization}

\author[Payne]{Sam Payne}

\begin{abstract}
We use functoriality of tropicalization and the geometry of projections of subvarieties of tori to show that the fibers of the tropicalization map are dense in the Zariski topology.  For subvarieties of tori over fields of generalized power series, points in each tropical fiber are obtained ``constructively" using Kedlaya's transfinite version of Newton's method.
\end{abstract}

\maketitle

\section{Introduction}

A well-known folk theorem says that the tropicalization map for subvarieties of tori over algebraically closed nonarchimedean fields is surjective.  This result has been published in various forms, both without proof in \cite[Theorem~9.14]{Sturmfels02}, and with proposed proofs in \cite[Theorem~2.1 and Corollary~2.2]{SpeyerSturmfels04} and \cite[Proposition~2.1.2]{SpeyerThesis}.  However, as several researchers have noted, there is a  critical gap in these published proofs.\footnote{In \cite{SpeyerSturmfels04}, the ring $(S^{-1}R_K[x]/S^{-1}(I \cap R_K[x]) ) \otimes_{R_K} K$ is isomorphic to a localization of $K[x]/I$, not to $K[x]/I$ itself, as implicitly claimed in the first sentence of the third paragraph of the proof of Theorem~2.1.  It is unclear why a maximal ideal in this localization, which corresponds to an ideal in $K[x]/I$ that is maximal among those whose intersection with $R_K[x]/I$ is contained in $\overline m$, necessarily corresponds to a maximal ideal in $K[x]/I$.  In \cite{SpeyerThesis}, the notation is different, but the missing step is essentially the same.  If $R_K$ was a discrete valuation ring then this gap could be filled by applying standard arguments from intersection theory to $\Spec R_K[x]/I$, as in \cite{Katz06}.  However, most $K$-varieties are not defined over discrete valuation subrings, so another argument is required.}

The idea of nonarchimedean amoebas and surjectivity of tropicalization seems to have origins in an unpublished manuscript of Kapranov, which treats the hypersurface case, and was eventually partially incorporated into the paper \cite{EKL}.  Since problems were noted in the published proofs of the general case, other versions of surjectivity of tropicalization have appeared in a rapid series of preprints.  Katz descends to a discrete valuation ring to show surjectivity onto rational points for varieties over fields of Puiseux series in characteristic zero \cite[Lemma~4.15]{Katz06}.  In the same setting, Jensen, Markwig, and Markwig give an implementable algorithm for producing approximations to a point in each rational fiber \cite{JMM}.  Another approach, using affinoid algebras and rigid analytic geometry, is developed in \cite[Theorem~4.2]{Draisma06}.

Here we prove a strong form of surjectivity of tropicalization over an abitrary algebraically closed nonarchimedean field, showing that the fibers of tropicalization are not only nonempty but Zariski dense.  The main new idea in the proof is to use the functorial properties of tropicalization to reduce to the hypersurface case via projections to smaller dimensional tori.

\section{Preliminaries}

\subsection{Fields}

Let $K$ be an algebraically closed field and let $\nu: K^* \rightarrow \R$ be a nonarchimedean valuation.  Say $R \subset K$ is the valuation ring for $\nu$, with $\m \subset R$ the maximal ideal and $k = R/\m$, and let $G \subset \R$ be the image of $\nu$.  Since $K$ is algebraically closed, the residue field $k$ is also algebraically closed and the valuation group $G$ is divisible.

\begin{example}
Suppose $k$ is an algebraically closed field and $G \subset \R$ is a divisible subgroup.  Then the field of  generalized power series $K = k((t^G))$ with coefficients in $k$ and exponents in $G$ is algebraically closed.  The elements of $K$ are the formal sums 
$a = \sum_{i \in G} a_i t^i$,
with $a_i$ in $k$, such that the support of $a$
\[
\supp (a) = \{ i \in G \ | \ a_i \neq 0 \}
\]
is well-ordered.  The map $\nu: K^* \rightarrow G$ taking a nonzero generalized power series to the minimal element of its support is a nonarchimedean valuation.

See the introductory sections of \cite{Poonen93} for an elegant overview of generalized power series rings and their history and \cite[pp.\ 598--602]{Passman77} for proofs that the na\"ive operations of addition and multiplication of formal sums are well-defined and that every nonzero element of $K$ has a multiplicative inverse.   Ribenboim proved that K is algebraically closed using the theory of maximal immediate extensions 
\cite{Ribenboim92}. Kedlaya gave another proof using a transfinite version of Newton's algorithm \cite{Kedlaya01}, which inspired some of the arguments in this paper. These results are related to earlier work on algebraic properties of generalized power series by Rayner \cite{Rayner68} and \c Stef\u anescu \cite{Stefanescu82, Stefanescu83}. See also \cite{McDonald95} and \cite{Kedlaya01b}. 
\end{example}

\noindent Throughout this note, we will return to the example $K = k((t^\R))$ of generalized power series with real exponents, and the reader may wish to keep this special case in mind.

\begin{remark}
 The main reason for working with generalized power series, instead of Puiseux series, is that fields of Puiseux series are never algebraically closed in positive characteristic. If $k$ has characteristic $p > 0$ then the solutions of the Artin-Schreier polynomial 
 $x^p - x - t^{-1}$ in the generalized power series field $k((t^\Q ))$ are 
\[
(t^{-1/p} + t^{-1/p^2} + t^{-1/p^3}+ \cdots ) + c,
\]
for $c \in \F_p$ , and none of these solutions lie in the Puiseux subfield. Of course, when $k$ is algebraically closed of characteristic zero, the Puiseux field $k \{\{t\}\}$ is algebraically closed with valuation group $G = \Q$, and the methods and results of this paper apply.
\end{remark}

\begin{remark}
Nonarchimedean fields other than Puiseux series and generalized power series occur naturally in many contexts. For instance, the nonarchimedean valuations on number fields and function fields of curves are used to construct ``adelic amoebas", which control nonexpansive sets in algebraic dynamics. See \cite[Section~4]{EKL} and \cite{adelicamoebas}. For such applications, it is essential to develop the 
basic properties of tropicalization for more general nonarchimedean fields. 
\end{remark}

\subsection{Tropicalization of tori and their subvarieties}
Let $T$ be an algebraic torus, with $M$ its lattice of characters.  Let $N = \Hom(M,\Z)$ be the dual lattice of one parameter subgroups, and let $N_G = \Hom(M,G)$, which is a subgroup of the real vector space $N_\R = \Hom(M, \R)$.  Recall that there is a natural bijection from $T(K)$ to the set of group maps $M \rightarrow K^*$ taking $x \in T(K)$ to the evaluation map
\[
\ev_x : M \rightarrow K^*,
\]
given by $u \mapsto \chi^u(x)$.  Composing $\ev_x$ with $\nu$ gives a group homomorphism $M \rightarrow G$, which we denote by $\Trop(x) \in N_G$.  This gives a tropicalization map
\[
\Trop: T(K) \rightarrow N_\R.
\]
For any subvariety $X$ of $T$, we write $\Trop(X)$ for the image of $X(K)$ under $\Trop$.  

\begin{example}
Suppose $K = k((t^\R))$.  Any choice of basis for the lattice of characters induces a splitting of the torus, $T \cong (K^*)^n$ and an isomorphism $N_\R \cong \R^n$.  In such coordinates, the tropicalization map from $(K^*)^n$ to $\R^n$ takes a tuple of nonzero generalized power series $(x_1, \ldots, x_n)$ to its vector of valuations $(\nu(x_1), \ldots, \nu(x_n))$.
\end{example}

\subsection{Functoriality of tropicalization}
Let $\varphi: T \rightarrow T'$ be a map of tori, and let $M'$ be the character lattice of $T'$.  Pulling back characters gives a natural map of lattices $\varphi^*: M' \rightarrow M$.  Let $\phi: N_G \rightarrow N'_G$ be the dual map on $\Hom$ groups.

\begin{proposition} \label{coarse functoriality}
If $\varphi: T \rightarrow T'$ is a map of tori, then
\[
\Trop \circ
\varphi = \phi \circ \Trop.
\]
\end{proposition}

\begin{proof}
Both maps take $x \in T(K)$ to the linear function taking $u \in M'$ to the valuation of $\chi^u(\varphi(x))$.\end{proof}

\begin{corollary}
If $\varphi: T \rightarrow T'$ maps $X$ into $X'$, then $\phi$ maps $\Trop(X)$ into $\Trop(X')$.
\end{corollary}

\noindent Furthermore, the map of $\Hom$ groups corresponding to a composition of maps of tori is the composition of the corresponding $\Hom$ group maps, so $\Trop$ is a functor from $K$-subvarieties of tori to subsets of $\Hom$ groups.

\subsection{Residue maps and exploded tropicalization}

The fiber of the tropicalization map over the origin is
\[
\Trop^{-1}(0) = T(R),
\]
since $T(R)$ is the subgroup of $T(K)$ consisting of points $x$ such that $\ev_x$ maps $M$ into $R^*$.  In other words, the tropicalization map fits into a short exact sequence of abelian groups
\[
0 \rightarrow T(R) \longrightarrow T(K) \xrightarrow{\Trop} N_G \rightarrow 0.
\]
For points in $T(R)$, composing the evaluation map $\ev_x$ from $M$ to $R^*$ with the residue map from $R^*$ to $k^*$ gives a natural map to $T(k)$.  For a point $x \in T(R)$, we write $\underline{x}$ for its image, or ``tropical residue", in $T(k)$.

If $v$ is a nonzero point in $N_G$, then $\Trop^{-1}(v)$ is a coset of $T(R)$ in $T(K)$.  In particular, $\Trop^{-1}(v)$ is a torsor\footnote{See \cite{Baez-torsors} for an excellent motivational introduction to torsors, with an emphasis on their simplicity and naturality in mathematics and physics.} over $T(R)$.  Furthermore, the action of $T(R)$ on $\Trop^{-1}(v)$ is algebraic, in the sense that $\Trop^{-1}(v)$ is naturally identified with the set of $R$-points of a $T$-torsor $T_v$, as follows.  

Let the ``tilted group ring" $R[M]^v$ be the subset of $K[M]$ consisting of those Laurent polynomials $\sum a_u x^u$ such that $\nu(a_u)$ is greater than or equal to $\<u,v\>$ whenever $a_u$ is nonzero.  It is straightforward to check that $R[M]^v$ is closed under multiplication and addition, and contains $R$.  We define
\[
T_v = \Spec R[M]^v.
\]
Then there is a natural map of $R$-algebras
\[
R[M]^v \rightarrow R[M] \otimes R[M]^v
\]
given by $a x^u \mapsto x^u \otimes a x^u$, corresponding to a natural simply transitive group action
of $T$ on $T_v$.  This gives $T_v$ canonically the structure of a $T$-torsor defined over $R$.

Now $T_v(R)$ is the set of $R$-algebra maps $R[M]^v \rightarrow R$.  Given a point $x$ in $\Trop^{-1}(v)$, we get a point in $T_v(R)$ by taking $a_u x^u$ to $a_u / \ev_x(u)$, and it is straightforward to check that the induced natural map
\[
\Trop^{-1}(v) \xrightarrow{\sim} T_v(R)
\]
is bijective.  Composing with the residue map from $T_v(R)$ to $T_v(k)$ then gives a canonical map taking a point $x$ in $\Trop^{-1}(v)$, to its tropical residue $\underline x$ in $T_v(k)$.

\begin{remark}
Since $K$ is algebraically closed, $T_v(R)$ is not empty, and for any point in $T_v(R)$ there is a unique isomorphism from $T_v$ to $T$ taking this point to the identity.  However, there is no canonical way to choose a point in each $T_v(R)$ in general, so to preserve the functoriality of our basic constructions we are obliged to work with the torsors themselves.
\end{remark}

\begin{example} \label{splitting}
The situation is somewhat simpler when there is a natural section of $\Trop$ that splits the short exact sequence above.  For instance, suppose $K = k((t^\R))$ and fix $v \in N_\R$.  Let $y$ be the point in $T(K)$ given by
\[
\ev_y(u) = t^{-\<u,v\>}.
\]
Then translation by $y$ maps $\Trop^{-1}(v)$ naturally and  bijectively onto $\Trop^{-1}(0)$, inducing a natural isomorphism $T_v \cong T$.  Therefore, in this case we may think of the residue map as taking $\Trop^{-1}(v)$ into $T(k)$.  If we choose coordinates by fixing an isomorphism identifying $T$ with $(K^*)^n$ then this residue map takes a tuple of nonzero generalized power series $(x_1, \ldots, x_n)$ to its tuple of leading coefficients $(\underline x_1, \ldots, \underline x_n)$.
\end{example}

\begin{definition}
The exploded tropicalization map is the natural map
\[
\TTrop: T(K) \rightarrow \coprod_v T_v(k)
\]
taking a point $x$ in $\Trop^{-1}(v)$ to its tropical residue $\underline x$ in $T_v(k)$.
\end{definition}

\noindent For any subvariety $X \subset T$, we write $\TTrop(X)$ for the image of $X(K)$ under the exploded tropicalization map.

\vspace{5 pt}

\begin{remark}
The exploded tropicalization $\TTrop(X)$ may be thought of as an algebraic analogue of the ``exploded torus fibrationsÓ studied by Parker from a symplectic veiwpoint \cite{Parker07, Parker08}. In particular, if the residue field $k$ is the field of complex numbers and each tropical reduction $X_v$ is smooth, then $\TTrop(X)$ is naturally identified with the set of points of an object in Parker's exploded category.
\end{remark}

Suppose $\varphi: T \rightarrow T'$ is a map of tori, and let $\phi: N_G \rightarrow N'_G$ be the induced group homomorphism.  Then it is straightforward to check that for any $v$ in $N_G$, the pullback under $\varphi$ of the tilted group ring $R[M']^{\phi(v)}$ is contained in $R[M]^v$ and that this induces a map of torsors 
\[
\varphi_v : T_v \rightarrow T'_{\phi(v)},
\]
over $\varphi$.

\begin{proposition}
If $\varphi: T \rightarrow T'$ is a map of tori, then
\[
\TTrop \circ \varphi =  \coprod_v \varphi_v \circ \TTrop.
\]
\end{proposition}

\begin{proof}
Let $x$ be a point in $\Trop^{-1}(v)$ whose image in $T_v(R)$ is given by
\[
\ev_x : R[M]^v \rightarrow R.
\]  
Then both maps take $x$ to the point $y$ in $T'_{\phi(v)}(k)$, where $\ev_y$ takes $f \in R[M']^{\phi(v)}$ to the residue of $\ev_x (\varphi^*f)$ in $R/\m$.  
\end{proof}

\begin{corollary}
If $\varphi: T \rightarrow T'$ maps $X$ into $X'$, then $\coprod_v \varphi_v$ maps $\TTrop(X)$ into $\TTrop(X')$.
\end{corollary}

\noindent It follows that the exploded tropicalization map $\TTrop$ is a functor from $K$-subvarie\-ties of tori to subsets of the set of $k$-points of a disjoint union of torus torsors.  It comes with a ``forgetful" natural transformation to the ordinary tropicalization functor $\Trop$ that takes points in $T_v(k)$ to $v$.

\section{Initial forms}

We briefly recall some basic ideas related to initial ideals, as developed in connection with tropical geometry by Speyer and Sturmfels \cite{SpeyerSturmfels04, SpeyerThesis}.  

Recall that $T_v$ is canonically defined over the valuation ring $R$, so the fiber over the closed point of $\Spec R$ is a scheme over the residue field $k = R/\m$.  The coordinate ring $k[T_v]$ of this scheme is a quotient of the tilted group ring $R[M]^v$, as follows.  For each real number $s$, let $\m^s$ (resp.\ $\m^s_+$) be the fractional ideal of $R$ whose nonzero elements are the elements of $K^*$ with valuation greater than or equal to $s$ (resp.\ strictly greater than $s$).  Let $k^s = \m^s / \m^s_+$.  Then $R[M]^v = \bigoplus_{u \in M} \m^{\<u,v\>}$, and $k[T_v]$ is the natural quotient
\[
k[T_v] = \bigoplus_{u \in M} k^{\<u,v\>}.
\]

Now each point $v \in N_\R$ determines a weight function on monomials in $K[M]$, given by
\[
b x^u \mapsto \nu(b) - \<u,v\>,
\]
for $u \in M$ and $b \in K^*$.  If all of the terms of a Laurent polynomial $f \in K[M]$ have nonnegative weight with respect to $v$, then $f$ is in the tilted group ring $R[M]^{v}$, and the initial form $\init_v(f)$ is defined to be the image of $f$ in $k[T_v]$.

\begin{definition} 
For any $K$-subvariety $X \subset T$, the tropical reduction $X_v$ is the $k$-subvariety of $T_v$ cut out by the initial forms of all Laurent polynomials in the intersection of $R[M]^v$ and the ideal of $X$.
\end{definition}

\noindent The tropical degeneration $X_v$ is closely related to the initial degenerations $\init_v(X)$ of subvarieties of affine space studied in Gr\"obner theory.  The notation and terminology here is meant to emphasize that $X_v$ lives inside the torus torsor $T_v$, and not in some partial compactification.

\begin{lemma} \label{reduction lemma}
If $x$ is a point in $\Trop^{-1}(v) \cap X(K)$, then $\underline x$ is in $X_v(k)$.
\end{lemma}

\begin{proof}
If $x$ is a point in $X(K)$ and $f \in K[M]$ is a Laurent polynomial whose monomials have nonnegative weight with respect to $v$, then $f(x)$ is in $R$, and $\init_v(f)(\underline x)$ is the image of $f(x)$ in $R/\m$.  In particular, if $f(x)$ vanishes then $\init_v(f) (\underline x)$ vanishes.
\end{proof}

\begin{example}
Suppose $K = k((t^\R))$.  Then each Laurent polynomial $f \in K[M]$ can be written uniquely as a formal sum of monomials $f = \sum_{u,i} a_{u,i} x^u t^i$, with coefficients $a_{u,i} \in k^*$.  Then the weight of a monomial $a_{u,i} x^u t^i$ is $i - \<u,v\>$, and it is customary to define the initial  form $\init_v(f) \in k[M]$ to be the sum of the monomials of lowest weight.  This agrees with the definition in the general case above after identifying $k[T_v]$ with $k[M]$, as in Example~\ref{splitting}, provided that the weight of the lowest weight monomials is zero.  All such initial forms of Laurent polynomials over in a given ideal in $K[M]$ is clearly in the ideal generated by those of weight zero, so the resulting degenerations are the same.
\end{example}

\section{Statement and applications of main result}

Roughly speaking, Lemma \ref{reduction lemma} says that if the valuation vector $\Trop(x)$ of a point $x \in X(K)$ is equal to $v$, then its leading coefficient vector $\underline x$ is in $X_v(k)$.  Surjectivity of tropicalization then says that any point in $X_v(k)$ can be lifted to a point in $\Trop^{-1}(v) \cap X(K)$.  Our main result is that the set of all such lifts is Zariski dense.

\begin{theorem} \label{main}
For any $v \in N_G$ and any $\underline x \in X_v(k)$,
\[
\TTrop^{-1}(\underline x) \cap X(K)
\]
is Zariski dense in $X(K)$.
\end{theorem}

\begin{corollary}
If $v$ is in $N_G$ and $X_v$ is nonempty then
\[
\Trop^{-1}(v) \cap X(K)
\]
is Zariski dense in $X(K)$.
\end{corollary}

\begin{remark}
Arguments similar to those used in the proof of~Theorem \ref{main} show that each nonempty fiber of $\Trop$ is a generically finite cover of the set of $R$-points of a $d$-dimensional torus, where $R \subset K$ is the valuation ring.  Similarly, if $K = k((t^\R))$ then each nonempty fiber of $\TTrop$ is a generically finite cover of a product of $d$ copies the maximal ideal $\m \subset R$.  Roughly speaking, this means that each nonempty fiber of the tropicalization of a $d$-dimensional subvariety of a torus is intrinsically $d$-dimensional and therefore should not be contained in any $(d-1)$-dimensional subvariety.  This motivates Theorem~\ref{main}.
\end{remark}

As applications of Theorem~\ref{main}, we have the following surjectivity statement for projections of tropicalizations.

\begin{corollary}
If $\varphi: T \rightarrow T'$ is a map of tori and $X'$ is the closure of $\varphi(X)$ then $\coprod_v \varphi_v$ maps $\TTrop(X)$ surjectively onto $\TTrop(X')$.
\end{corollary}

\noindent Specializing to the ordinary tropicalization map, we recover the following surjectivity result due to Tevelev \cite[Proposition~3]{Tevelev05}.\footnote{ Tevelev stated this result only for the algebraic closures of Puiseux fields, but his proof goes through in full generality.}

\begin{corollary}
If $\varphi: T \rightarrow T'$ is a map of tori and $X'$ is the closure of $\varphi(X)$ then $\phi: N_G \rightarrow N'_G$ maps $\Trop(X)$ surjectively onto $\Trop(X')$.
\end{corollary}

\section{Projections of subvarieties of tori}

Let $T$ be an $n$-dimensional torus.  Let $d$ be a nonnegative integer, and let $G_d$ be the Grassmannian of $d$-dimensional subspaces of $N_\Q$.

Say that a projection of $T$ is a split surjection $T \rightarrow T'$.  Since a projection is determined (up to canonical isomorphism) by its kernel, the set of projections of $T$ onto $(n-d)$-dimensional tori is naturally identified with $G_d(\Q)$.  The correspondence takes a projection to the subspace of $N_\Q$ spanned by the lattice of characters of its kernel.  We consider $G_d(\Q)$ with the classical topology induced by its inclusion in the manifold $G_d(\R)$, the Grassmannian of real $d$-dimensional subspaces of $N_\R$.

Recall that for a map of tori $\varphi: T \rightarrow T'$, we write $\phi: N_\R \rightarrow N'_\R$ for the induced map on vector spaces spanned by lattices of one-parameter subgroups.

\begin{proposition} \label{general projection}
Suppose $X \subset T$ is a subvariety of codimension $d+1$ and the tropical degeneration $X_0$ contains the identity $1_T$.  Then the set of projections $\varphi: T \rightarrow T'$ with the following three properties is dense in $G_d(\Q)$.
\begin{enumerate}
\item The closure of $\varphi(X)$ is a hypersurface in $T'$.
\item For all nonzero $v$ in the kernel of $\phi$, $X_v$ is empty.
\item The intersection of $X_0$ with the kernel of $\varphi$ is $\{1_T\}$.
\end{enumerate}
\end{proposition}

\noindent In the proof of the proposition we will use the following basic lemma.  Let $X' \subset T'$ be the closure of $\varphi(X)$.

\begin{lemma} \label{functoriality of SS}
For any $v \in N_\R$, $\varphi_v(X_v)$ is contained in $X'_{\phi(v)}$.
\end{lemma}

\noindent In particular, if $X_v$ is nonempty then $X'_{\phi(v)}$ is nonempty.

\begin{proof}
Suppose $f$ is in the ideal of $X'$.  Then $\varphi_v^* \init_{\phi(v)}(f)$ is equal to $\init_v (\varphi^*(f))$, and $\varphi^* f$ vanishes on $X$.  Hence $\varphi^* \init_{\phi(v)}(f)$ vanishes on $X_v$, as required.
\end{proof}

\begin{proof}[Proof of Proposition \ref{general projection}]
Let $S(X) \subset N_\R$ be the set of points $v$ such that $X_v$ is not empty.  Then $S(X)$ is the underlying set of a finite polyhedral complex of pure dimension equal to the dimension of $X$ \cite[Theorem~9.6]{Sturmfels02}.  See also \cite[Proposition~2.4.5]{SpeyerThesis}.\footnote{The polyhedral structure and pure dimensionality of $S(X)$ is closely related to the Bieri-Groves Theorem that the set of characters induced by valuations on the function field $K(X)$ that extend the valuation $\nu$ gives a polyhedral complex in $N_\R$ of pure dimension $d$ \cite{BieriGroves84}, a special case of which proved Bergman's conjecture on logarithmic limit sets of varieties \cite{Bergman71}.}  Then projection along a general $d$-dimensional subspace of $N_\Q$ maps $S(X)$ onto the underlying set of a finite polyhedral complex of pure codimension one.    By Lemma~\ref{functoriality of SS}, $\phi(S(X))$ is contained in $S(\overline{\varphi(X)})$, and it follows that if $\phi(S(X))$ has codimension one then $\overline{\varphi(X)}$ must be a hypersurface.  This shows that (1) and (2) hold for a dense open subset of $G_d(\Q)$.

We now prove that (3) holds on a dense subset of $G_d(\Q)$ by induction on the dimension of $X$.  Suppose $X$ is zero-dimensional.  Then $X_0$ has dimension zero, and we must show that a general codimension one subtorus misses the zero-dimensional subscheme $X_0 \smallsetminus 1_T$.  Let $x \neq 1_T$ be a point in $X_0$.  Then the kernel of the evaluation map $\ev_x : M \rightarrow K^*$ is a proper sublattice
\[
M_x \subset M,
\]
and a subtorus $T_d \subset T$ contains $x$ if and only $M_x$ contains $N_d^\perp \cap M$, where $N_d$ is the character lattice of $T_d$.  Now the complement of any finite union of proper sublattices contains a translate of a sublattice of full rank.  In particular, there is a translate of a sublattice of full rank $\Lambda \subset M$ that is disjoint from $M_x$ for every $x$ in $X_0 \smallsetminus 1_T$.  If $\ell$ is the line through a point in $\Lambda$, then $\ell^\perp$ corresponds to a codimension one subtorus disjoint from $X_0 \smallsetminus 1_T$.  Since $\Lambda$ is a translate of a sublattice of full rank, the set of all such $\ell^\perp$ is dense (though not open) in $G_d(\Q)$.

Now, suppose $X$ is positive dimensional.  For any codimension one torus $T_1 \subset T$ not containing $X$, $T_1 \cap X$ has pure codimension $d$ in $T_1$, and hence dimension strictly smaller than $X$.  Then, by induction on dimension, there is a dense set of $d$-dimensional subtori $T_0 \subset T_1$ such that the projection $T_1 \rightarrow T_1/ T_0$ satisfies (3) for $T_1 \cap X$.  If $T_0$ is such a torus, then the projection $T \rightarrow T/T_0$ also satisfies (3).

Therefore, the set of projections satisfying all three conditions is the intersection of an open dense subset of $G_d(\Q)$ where (1) and (2) are satisfied with a dense set where (3) is satisfied, and hence is dense.
\end{proof}

\section{Proof of main result}

Let $\underline x$ be a point in $X_v(k)$, for some $v \in N_G$.  We must show that $\TTrop^{-1}(\underline x) \cap X(K)$ is Zariski dense in $X(K)$.  After translating by a point in $\Trop^{-1}(-v)$, we may assume that $v = 0$ and $\underline x = 1_T$.

\subsection{Reduction to the hypersurface case} \label{reduction to hypersurface}

Choose a projection $\varphi: T \rightarrow T'$ as in Proposition \ref{general projection}.  Let $X' \subset T'$ be the closure of the image of $X$.  By Lemma~\ref{functoriality of SS}, $1_{T'}$ is in $X'_0$.

Suppose $\TTrop^{-1}(1_{T'})$ is dense in $X'(K)$.  By the choice of the projection, any point in $X(K)$ that maps into $\TTrop^{-1}(1_{T'})$ must lie in $\TTrop^{-1}(1_T)$.  Now $\varphi(X)$ contains an open dense subvariety of $X'$, so it follows that the image of $\TTrop^{-1}(1_T)$ is dense in $X'(K)$.  Since $X$ and $X'$ have the same dimension, $\TTrop^{-1}(1_T)$ must be dense in $X(K)$.

Therefore, we may assume that $X$ is a hypersurface.

\subsection{Tropicalization of hypersurfaces} \label{tropicalization of hypersurfaces}

Suppose $X \subset T$ is a hypersurface.  Let $f \in K[M]$ be a defining equation for $X$.

\begin{proposition} \label{hypersurface}
If $X \subset T$ is a positive dimensional hypersurface with $\underline x \in X_0$ then $\TTrop^{-1}( \underline x)$ is infinite.
\end{proposition}

\begin{proof}
We prove the proposition by explicitly producing solutions to $f$ with the desired tropicalization using Newton's method, as in Kedlaya's proof that fields of generalized power series over an algebraically closed field are algebraically closed \cite{Kedlaya01}.

Choose a splitting $M \xrightarrow{\sim} \Z^n$, inducing an identification
\[
K[M] \cong K[ y_1 , y_1^{-1}, \ldots, y_n, y_n^{-1}].
\]
After multiplying $f$ by a monomial $a x^u$ with $a \in K^*$ and $u \in M$, we may assume that $f \in R[y_1, y_2, \ldots, y_n]$ is a polynomial and that $\init_0(f)$ is nonzero.

Let $\underline x = (\underline x_1, \ldots, \underline x_n)$ be a solution of $\init_0(f)$.  Choose lifts $x_2, \ldots, x_n$ in $R^*$, with residues $\underline x_2, \ldots, \underline x_n$, respectively.  Let $g$ be the polynomial $g(y) = f (y, x_2, \ldots, x_n)$.
Since $K$ is algebraically closed, $g$ factors as a product of linear terms
\[
g(y) = y^b (c_1 y - d_1) \cdots (c_r y - d_r),
\]
which can be normalized so that each $c_i$ and $d_i$ are in $R$.  Now, initial forms are multiplicative, so $\init_0(g) = y^b \init_0(c_1y - d_1) \cdots \, \init_0(c_r y - d_r)$.  By construction, $\underline x_1$ is a solution of $\init_0(g)$.  Therefore, there is a solution $y = d_i / c_i$ of $g$ in $R^*$ with residue $\underline x_1$.  In particular, since there are infinitely many distinct lifts of $\underline x_2$, this produces infinitely many points in $\TTrop^{-1}(\underline x) \cap X(K)$, as required.
\end{proof}

\begin{example}
If $K = k((t^\R))$ then solutions to $g$ in the proof of Proposition~\ref{hypersurface} with valuation zero and the desired residue $\underline x_1$ can be constructed by Kedlaya's transfinite version of Newton's method.  See \cite[Section~2]{Kedlaya01}, or the Appendix below.  This approach is closely related to the algorithmic arguments developed independently by Jensen, Markwig, and Markwig \cite{JMM}.
\end{example}

\subsection{Conclusion of proof}

\begin{proof}[Proof of Theorem \ref{main}]  
We may assume that $X$ is a hypersurface and $\underline x$ is in $X_0(k)$.  Let $y \in X(K)$ be a general point.  We must show that $y$ is in the Zariski closure of $\TTrop^{-1}(\underline x)$.  By Proposition~\ref{hypersurface}, $\TTrop^{-1}(\underline x) \cap X(K)$ is nonempty, so we may choose a point $x \in \TTrop^{-1}(\underline x)$.  Furthermore, since $X$ is affine, we can find a curve $C \subset X$ containing both $x$ and $y$.  Therefore, it will suffice to show that $y$ is in the Zariski closure of $\TTrop^{-1}(\underline x) \cap C(K)$.  In other words, it suffices to prove Theorem~\ref{main} for an arbitrary curve $C$.

Once again, by projecting as in Proposition~\ref{general projection}, we may assume that $C$ is a curve in a two-dimensional torus.  Then Proposition~\ref{hypersurface} says that $\TTrop^{-1}(\underline x) \cap C(K)$ is infinite, and hence Zariski dense, as required.
\end{proof}

\section{Appendix: Kedlaya's transfinite version of Newton's Algorithm}

Here we give a strong lifting property for solutions of polynomials over generalized power series.  This lifting property appears implicitly in Kedlaya's proof that such fields are algebraically closed if the residue field is algebraically closed.  For the reader's convenience, we sketch a ``tropical" proof, adapted from Kedlaya's original presentation of the transfinite version of Newton's algorithm in terms of slopes of faces of Newton polytopes \cite[Section~2]{Kedlaya01}.  Here, $K = k((t^G))$ is an algebraically closed field of generalized power series. 

\begin{KL}
 Let $f \in K[z,z^{-1}]$ be a Laurent polynomial in a single variable, and let $a \in k^*$ be a solution of $\init_v (f)$ with multiplicity $m$.  Then there are exactly $m$ solutions $f(x) = 0$ in $K^*$ such that $\nu(x) = v$ and $\underline x = a$, counted with multiplicity.
\end{KL}

\begin{proof}
After multiplying $f$ by a power of $z$, we may assume that 
\[
f = c_0 + c_1 z + \cdots + c_d z^d
\]
is a polynomial with nonzero constant term.  By induction on the degree of $f$, we may assume that $v$ is the largest real number such that $\init_v(f)$ is not a monomial, and replacing $z$ by some $z + c$, we may assume that $v = 0$.  Furthermore, after multiplying both $f$ and $z$ by suitable powers of $t$, we may assume that $\nu(c_i)$ is nonnegative for all $i$, and $\nu(c_0) = 0$.

We inductively construct a transfinite sequence of approximations $x_\omega \in K^*$ converging to the required solution of $f$, as follows.  Set
\[
x_1 = a, \mbox{ and } y_1 = f(x_1).
\]
Since $x_1$ is a solution of $\init_0(f)$, $\nu(y_1)$ is strictly positive.  Now, for any $h \in K$, let $\Delta(h)$ be $f(z+h) - f(z)$.  Then $\Delta$ is a polynomial of degree $d$ with no constant term, so it may be written as
\[
\Delta(h) = c'_1h + \cdots + c'_d h^d,
\]
and each nonzero $c'_i$ has $\nu(c'_i) \geq 0$.  Let $\Psi$ be the ``tropicalization" of $\Delta$, which is the convex piecewise linear function on $\R$ given by
\[
\Psi(r) = \min \{ \nu(c'_i) + j \cdot r \ | \ c'_i \neq 0 \}.
\]
Then $\Psi$ is strictly increasing, so there is a unique $v_1$ such that $\Psi(v_1) = \nu(y_1)$.  The coefficient of $t^{\nu(y_1)}$ in $\Delta(a t^{v_1})$ is a nonzero polynomial in $a$ with no constant term, so it takes on every value in $k^*$, and when this coefficient is nonzero then the valuation of $\Delta(a t^{v_1})$ is $\nu(y_1)$.  In particular, we can choose $a_1$ such that the leading term of $\Delta(a_1t^{v_1})$ is $-\underline y_1 t^{\nu(y_1)}.$

Therefore, if we set 
\[
x_2 = a + a_1t^{v_1}
\]
then either $f(x_2) = 0$ or $\nu(f(x_2))$ is strictly greater than $v_1$.  We set $y_2 = f(x_2)$ and continue to produce successive approximations similarly, by transfinite induction, as follows.

For any ordinal $\omega$, suppose we have coefficients $a_j \in k^*$ and a strictly increasing sequence of exponents $v_j \in \R^*$, for $j < \omega$ with the property that if we set
\[
x_{j} = a + \sum_{j' < j} a_{j'} t^{v_{j'}} \mbox{ and } y_j = f(x_j),
\]
then $\nu (y_j)$ is greater than $\nu(y_{j'})$ for all $j' < j < \omega$.  Now, let $x_\omega = a + \sum_{j < \omega} a_j x^{v_j}$.  Then, it is clear that either $f(x_\omega) = 0$ or $\nu(f(x_\omega))$ is greater than $v_j$ for all $j < \omega$.  If $f(x_\omega)$ is not zero, then we set $y_\omega = f(x_\omega)$ and continue as in the case $\omega = 1$ above, choosing $v_\omega$ such that $\Psi(v_\omega) = \nu(y_\omega)$  and choosing $a_\omega$ such that the leading term of $\Delta(a_\omega t^{v_\omega})$ is $- \underline y_\omega t^{\nu(y_\omega)}$.  Then either $x_{\omega + 1} = x_\omega + a_\omega t^{v_\omega}$ is a solution for $f$ or $f(x_{\omega + 1})$ has valuation strictly greater than $\nu(y_\omega)$.

Since the increasing sequence of real exponents  $v_j$ is well-ordered, this sequence must be at most countable.  Therefore, the transfinite Newton's algorithm must terminate, producing a solution $f(x_\omega) = 0$, for some $\omega$ less than or equal to the first uncountable ordinal, with leading term $\underline x_\omega = a$.  That the number of such solutions is equal to $m$ is clear if $d = 1$, and follows in general by induction on degree.
\end{proof}

\vspace{5 pt}

\noindent \textbf{Acknowldegments}  While learning about the rapidly developing subject of tropical geometry I have benefited from many helpful discussions with friends and colleagues including W.~Fulton, P.~Hacking, E.~Katz, H.~Markwig, G.~Mikhalkin, E.~Shustin, J.~Tevelev, R.~Vakil, and J.~Yu.  The problem of understanding fibers of tropicalization was discussed at a one-day workshop at MSRI, and I thank the organizers B.~Sturmfels and E.~Feichtner, as well as the other participants, for their shared insights and encouragement. I am especially grateful to D.~Speyer and K.~Kedlaya for useful comments and corrections on an earlier draft of this note, B.~Osserman, from whom I first learned about the problem of surjectivity of tropicalization, and D.~Savitt, who brought Kedlaya's paper \cite{Kedlaya01} to my attention.  My thanks also to the referee for several corrections and helpful suggestions.  This work was supported by the Clay Mathematics Institute.

\bibliography{math}
\bibliographystyle{amsalpha}

\end{document}